\documentclass{article}

\usepackage[utf8]{inputenc}
\usepackage[english]{babel}
\usepackage{amsmath, amsthm, amssymb}
\usepackage[all]{xy}
\usepackage{centernot}
\usepackage{geometry}

\newtheorem{theoremcounter}{Theorem Counter}[section]

\theoremstyle{definition}
\newtheorem{definition}[theoremcounter]{Definition}
\newtheorem*{remark}{Remark}
\newtheorem{observation}[theoremcounter]{Observation}

\theoremstyle{plain}
\newtheorem{lemma}[theoremcounter]{Lemma}
\newtheorem{proposition}[theoremcounter]{Proposition}
\newtheorem{corollary}[theoremcounter]{Corollary}
\newtheorem{theorem}[theoremcounter]{Theorem}

\newcommand{\loz}{\L{}o\'s}
\newcommand{\wrt}{w.r.t.\ }
\newcommand{\ie}{i.e.\ }

\newcommand{\N}{\mathbb{N}}
\newcommand{\calU}{\mathcal{U}}
\newcommand{\calF}{\mathcal{F}}
\newcommand{\frakA}{\mathfrak{A}}
\newcommand{\frakB}{\mathfrak{B}}
\newcommand{\frakC}{\mathfrak{C}}

\newcommand{\powfin}{\mathcal{P}_\mathrm{fin}}

\renewcommand{\phi}{\varphi}
\newcommand{\abs}[1]{\left\lvert#1\right\rvert}

\newcommand{\FO}{\mathrm{FO}}

\title{Ultraproducts as a tool for first-order inexpressibility in the finite and infinite}
\author{Philip Dittmann}
\date{\today}

\begin{document}

\maketitle

\begin{abstract}
Ultraproducts are a well-known tool in the classical model theory of first-order logic. We explore their uses in the context of finite model theory. 
\end{abstract}

\section*{Introduction}

It is well-known that many methods from the classical model theory of first-order logic, most importantly compactness, fail when we restrict attention to finite structures. In this paper we explore the applicability of ultraproducts -- a familiar construction in (infinite) model theory -- to known inexpressibility results in finite and infinite model theory. This complements the more commonly used approach through game techniques. In particular, we show new proofs for variants of the well-known locality theorems of Hanf and Gaifman with little or no use of games.

The idea of using infinitary methods in general and ultraproducts in particular for questions in finite model theory was explored by Väänänen in \cite{vaananen}, unknown to me at the time of writing. Some of our results are close to those in the preprint \cite{lindellWeinstein} by Lindell, Towsner and Weinstein.

\subsection*{Acknowledgements}

This paper is a condensed and revised version of my Bachelor's thesis at Technische Universität Darmstadt from June 2013. I am deeply indebted to my supervisor Martin Otto for helpful ideas, advice and corrections. I would also like to thank Steven Lindell, Henry Towsner and Scott Weinstein for a preprint of their paper \cite{lindellWeinstein}.

\section{General results for ultraproducts}

We assume that the reader is familiar with the ultraproduct construction and \loz's theorem.
In this section we review some results about the structure of ultraproducts. Some of their distinctive properties make them useful for inexpressibility results in the upcoming sections.

\begin{proposition}\label{propositionCardinalityOfUltraproducts}
	Let $(A_i)_{i \in \N}$ be a sequence of non-empty sets and consider them as structures over the empty signature. Let furthermore $\calU$ be an arbitrary non-principal ultrafilter on $\N$. Then the cardinality of the ultraproduct $\prod_i A_i / \calU$ is finite iff there is an $n \in \N$ such that $\abs{A_i} \leq n$ for $\calU$-many $i$. If the cardinality is not finite, then it is at least $2^\omega$.
\end{proposition}
\begin{proof}
	The ultraproduct has at most $n$ elements iff $\calU$-many factors have at most $n$ elements, since this property can be expressed by a first-order formula. 
	Now assume conversely that for each $n \in \N$ we have $\calU$-many sets $A_i$ with more than $n$ elements. We show that there is an injective embedding of $\{ 0, 1 \}^\N$ into $\prod_i A_i / \calU$. To do this, we partition the index set $\N$ into disjoint sets $I_k$, $k \in \N$, such that $I_k \notin \calU$ and $\abs{A_i} \geq 2^k$ for all $i \in I_k$. This is always possible: 
	
	If $\calU$-many $A_i$ are finite, we simply put \[ I_k := \{ i \in \N \colon 2^k \leq \abs{A_i} < 2^{k + 1} \} \] for $k > 0$ and \[ I_0 := \{ i \in \N \colon \text{$A_i$ is infinite or has exactly one element } \} ;\] if $\calU$-many $A_i$ are infinite, we gather all $i$ such that $A_i$ is finite in $I_0$ and make all other $I_k$ one-element sets such that we eventually enumerate all $i$.
	
	Having found such a partition into sets $I_k$, we define the map
	\begin{align*}
		f \colon \{ 0, 1 \}^\N &\to \prod_i A_i \\
		 f\big( (c_n)_n \big)_i &= \sum_{l = 0}^{k - 1} c_l 2^l \in \{ 0, \dotsc, 2^k - 1 \} \subseteq A_i \text{ if $i \in I_k$} ,
	\end{align*}
	where we identify $\{ 0, \dotsc, 2^k - 1 \}$ with a subset of $A_i$. We want to show that $f$ composed with the canonical projection $\pi \colon \prod_i A_i \to \prod_i A_i / \calU$ is still injective. Let $(c_n)$ and $(d_n)$ be two $0$-$1$-sequences differing at some index $m \in \N$. Then $f\big((c_n)_n\big)_i \neq f\big((d_n)_n\big)_i$ for all $i \in I_k$ with $k > m$, \ie $f\big((c_n)_n\big)$ differs from $f\big((d_n)_n\big)$ at $\calU$-many indices and hence $\pi\big(f\big((c_n)_n\big)\big) \neq \pi\big(f\big((d_n)_n\big)\big)$. This proves that \[ \pi \circ f \colon \{ 0, 1 \}^\N \to \prod_i A_i / \calU \] is injective. 
\end{proof}

\begin{remark}
	One could also use a simpler diagonalisation argument in the spirit of Cantor to show that the ultraproduct cannot be countably infinite. However, we often do not want to assume the continuum hypothesis and will need the full strength of the proposition.
\end{remark}

\begin{corollary}
	Let $(A_i)_{i \in \N}$ be a sequence of non-empty sets of cardinality at most $2^\omega$ and $\calU$ a non-principal ultrafilter on $\N$. Then the ultraproduct $\prod_i A_i / \calU$ is either finite or its cardinality is exactly $2^\omega$.
\end{corollary}
\begin{proof}
	The set $\prod_i A_i$ has at most cardinality $(2^\omega)^\omega = 2^\omega$. Hence its image $\prod_i A_i / \calU$ under the canonical projection has cardinality at most $2^\omega$.
\end{proof}

For the investigation of other ultraproduct properties we use the familiar model-theoretic notion of \emph{types}. In the following we will always use partial $1$-types with parameters: a type of a structure $\frakA$ (with parameters from $X \subseteq A$) is therefore a set $\Phi \subseteq \FO_1(\sigma \mathop{\dot\cup} \{ c_x \colon x \in X \})$, where the signature is expanded by new constant symbols standing for elements of $X$, such that there is an elementary extension $\frakB \succeq \frakA$ and an element $b \in B$ with $\frakB, b \models \Phi$ where $\frakB$ is understood to interpret $c_x$ as $x \in A \subseteq B$ for all $x \in X$. 
We recall that a set of formulae $\Phi$ is a type of $\frakA$ iff it is finitely realised, \ie if for all finite $\Phi_0 \subseteq \Phi$ we have \[ \frakA \models \exists x \bigwedge \Phi_0(x) .\]

We also recall the definition of saturation and an important property of saturated structures.
\begin{definition}
	Let $\frakA$ be a $\sigma$-structure and $\lambda$ an infinite cardinal. If every type of $\frakA$ with fewer than $\lambda$ parameters is realised in $\frakA$, we call $\frakA$ \emph{$\lambda$-saturated}. If $\frakA$ is $\abs A$-saturated, we simply call $\frakA$ \emph{saturated}.
\end{definition}
\begin{proposition}[{\cite[Theorem 8.1.8]{hodges}}] \label{propElemEquivSatStrAreIsom}
	Let $\frakA$ and $\frakB$ be $\sigma$-structures of the same cardinality $\lambda$ such that both $\frakA$ and $\frakB$ are ($\lambda$-)saturated. If $\frakA$ and $\frakB$ are elementarily equivalent, then they are isomorphic.
\end{proposition}

We now show that ultraproducts are $\omega_1$-saturated in important cases.
\begin{proposition}\label{propCountableTypesAreRealised}
	Let $(\frakA_i)_{i \in \N}$ be a sequence of $\sigma$-structures, $\calU$ a non-principal ultrafilter on $\N$ and $\Phi$ a type of $\frakA :=\prod_i \frakA_i / \calU$ with $\abs\Phi \leq \omega$. Then $\Phi$ is realised in $\frakA$. Consider the following additional condition:
	\[ \text{For all $m \in \N$ and all finite subsets $\Phi_0 \subseteq \Phi$ it holds that $\frakA \models \exists^{\geq m} x \bigwedge \Phi_0(x)$.} \tag{$\ast$} \] 
	If condition $(\ast)$ holds, then there are at least $2^\omega$ elements of $\frakA$ realising $\Phi$; otherwise, the number of elements of $\frakA$ realising $\Phi$ is equal to $m_\text{max}$, the maximal $m \in \N$ for which $(\ast)$ holds for all finite subsets $\Phi_0 \subseteq \Phi$, and therefore finite.
\end{proposition}
\begin{proof}
	Since we can treat any parameters as additional constants in the signature, it suffices to consider $\Phi$ without parameters. Enumerate $\Phi$ as $(\phi_n)_{n \in \N}$. We now define new formulae $\psi_n$. If condition $(\ast)$ holds, let \[ \psi_n := \exists^{\geq n}x  \bigwedge_{k < n} \phi_k(x). \] Otherwise, let \[ \psi_n := \exists^{\geq m_\text{max}} x \bigwedge_{k < n} \phi_k(x). \] In both cases we have $\psi_0 \equiv \top$ by definition of the empty conjunction. For each $i \in \N$ we choose the maximal $m_i \leq i$ such that $\frakA_i \models \psi_{m_i}$. For all $n \in \N$ we now have $\frakA \models \psi_n$ and therefore $\calU$-many $i \geq n$ with $\frakA_i \models \psi_n$ and hence $m_i \geq n$. Let now $a := [(a_i)_i] \in A$ such that $\frakA_i, a_i \models \bigwedge_{k < m_i} \phi_k$. Then \loz's theorem implies that $\frakA, a \models \bigwedge_{k < n} \phi_k$ for all $n$; hence $a$ realises $\Phi$ in $\frakA$.
	
	The number of witnesses that we get in this manner depends on our choice of $\psi_n$. If condition $(\ast)$ holds and therefore $\psi_n = \exists^{\geq n}x  \bigwedge_{k < n} \phi_k(x)$, then $\frakA_i$ has at least $m_i$ witnesses for $\bigwedge_{k < m_i} \phi_k(x)$. Now Proposition \ref{propositionCardinalityOfUltraproducts} implies that we have constructed at least $2^\omega$ realisations of $\Phi$. But if condition $(\ast)$ is violated by $m_\text{max} + 1$ and some $\Phi_0 \subseteq \Phi$, clearly $\Phi$ itself cannot be realised more than $m_\text{max}$ times. On the other hand, all $\frakA_i$ have at least $m_\text{max}$ many witnesses for $\bigwedge_{k < m_i} \phi_k(x)$. This yields $m_\text{max}$ realisations of $\Phi$ in $\frakA$.
\end{proof}

\begin{corollary}
	Let $(\frakA_i)_{i \in \N}$ be a sequence of structures over a countable signature $\sigma$ and $\calU$ a non-principal ultrafilter on $\N$. Then the ultraproduct $\prod_i \frakA_i / \calU$ is $\omega_1$-saturated.
\end{corollary}
\begin{proof}
	Let $\Phi$ be a type with countably many parameters. The signature remains countable when we add constant symbols for these parameters. Since the set of first-order formulae over a countable signature is itself countable, $\Phi$ must be countable as well. By the preceding proposition, $\Phi$ is realised in the ultraproduct.
\end{proof}

\begin{corollary}\label{corollaryElemEquivUltraprodsAreIsomAssumingCH}
	Assume the continuum hypothesis and let $(\frakA_i)_{i \in \N}$ and $(\frakB_i)_{i \in \N}$ be two sequences of structures over a countable signature $\sigma$ such that each individual structure has cardinality at most $2^\omega$. Let furthermore $\calU_1$ and $\calU_2$ be two non-principal ultrafilters on $\N$ and consider the ultraproducts $\frakA := \prod_i \frakA_i / \calU_1$ and $\frakB := \prod_i \frakB_i / \calU_2$. Then these ultraproducts are isomorphic iff they are elementarily equivalent.
\end{corollary}
\begin{proof}
	Both ultraproducts have cardinality at most $2^\omega = \omega_1$. Hence Proposition \ref{propElemEquivSatStrAreIsom} is applicable, so the ultraproducts are already isomorphic if they are elementarily equivalent.
\end{proof}

\begin{remark}
	This corollary shows that for $\N$-fold ultraproducts some usually distinct notions of equivalence between structures collapse; we will see one more example of this in the next section. In the other sections we will prove isomorphy of some ultraproducts without requiring the continuum hypothesis.
\end{remark}

In the remainder of the section we will show that Corollary \ref{corollaryElemEquivUltraprodsAreIsomAssumingCH} may fail without the continuum hypothesis. This is related to work by Shelah (\cite{viveLaDifference}), although our counterexamples have been constructed without reference to it.

	We first construct a non-saturated ultraproduct. Consider the signature $\sigma = \{ E \}$ with a single binary relation and a two-sorted structure $\frakA$ given by $\omega_1$ and the set of its finite subsets, \ie \[ A := (\{ 0 \} \times \omega_1) \cup (\{ 1 \} \times \powfin(\omega_1)) \] with the relation interpreted by \begin{multline*} (a, b) \in E^\frakA :\iff \\ \text{ there exist $a_0 \in \omega_1, b_0 \in \powfin(\omega_1)$ such that $a = (0, a_0)$, $b = (1, b_0)$ and $a_0 \in b_0$}, \end{multline*} that is, we take $E$ to simply stand for elementhood. Now we consider a non-principal ultrafilter $\calU$ on $\N$ and form the ultrapower $\frakA_\infty := \frakA^\N / \calU$. Let \[ \iota \colon A \to A_\infty, a \mapsto [(a)_{n \in \N}] \] be the diagonal embedding. Consider the type \[ \Phi := \{ c_a E x \colon a \in \iota(\{ 0 \} \times \omega_1) \} \] where $c_a$ is the constant representing the parameter $a$. We easily see that $\Phi$ is finitely realised in $\frakA_\infty$ since for any finite subset $\Phi_0 \subseteq \Phi$ we can construct a witness componentwise. But for every $b = [ (b_i)_{i \in \N} ] \in A_\infty$ the set \[ I := \{ a \in \{ 0 \} \times \omega_1 \colon \text{ there exists an $i \in \N$ such that $a E^\frakA b_i$} \} \] is countable (since the set of all $a$ satisfying $a E^\frakA b_i$ for some fixed $i$ is finite); therefore its subset \[ \{ a \in \{ 0 \} \times \omega_1 \colon \iota(a) E^{\frakA_\infty} b \} \] is also countable and therefore not all of $\{ 0 \} \times \omega_1$, implying that $\Phi$ is not realised in $\frakA_\infty$. The type $\Phi$ has exactly $\omega_1$ parameters, so $\frakA_\infty$ is not $\omega_2$-saturated. However, $\frakA_\infty$ has cardinality $2^\omega$, which is at least $\omega_2$ if we assume the negation of the continuum hypothesis. Therefore $\frakA_\infty$ is not saturated under this assumption.

\begin{proposition}\label{propositionNonIsomElemEquivUltraprods}
	It is consistent with ZFC to have two $\N$-fold ultraproducts, both of cardinality $2^\omega$, which are elementarily equivalent but not isomorphic.
\end{proposition}
\begin{proof}
	By Easton's theorem \cite{easton} it is compatible with ZFC to require $2^{\omega_2} = 2^\omega$. Consider the structure $\frakA$ from the observation above. By Theorem 8.2.1 in \cite{hodges} we can find an $\omega_2$-saturated elementary extension $\frakB$ of $\frakA$ of cardinality no greater than \[ \omega_1^{< \omega_2} \leq \omega_1^{\omega_2} \leq (2^\omega)^{\omega_2} = 2^{\omega \omega_2} = 2^{\omega_2} = 2^\omega. \]
	Let now $\calU_1$ and $\calU_2$ be arbitrary (not necessarily distinct) non-principal ultrafilters on $\N$ and consider the ultrapowers $\frakA_\infty = \frakA^\N / \calU_1$ and $\frakB_\infty = \frakB^\N / \calU_2$. Assume that there exists an isomorphism $f \colon \frakA_\infty \to \frakB_\infty$. We know that the type
	\[ \Phi := \{ c_a E x \colon a \in \iota(\{ 0 \} \times \omega_1) \} \]
	of $\frakA_\infty$ is not realised in $\frakA_\infty$. If we can show that there is a $b$ in $\frakB_\infty$ that realises
	\[ \Phi' := \{ c_{f(a)} E x \colon a \in \iota(\{ 0 \} \times \omega_1) \}, \]
	we have a contradiction, since then $f^{-1}(b)$ would necessarily realise $\Phi$ in $\frakA_\infty$. But finding $b$ is easy: For every $a \in \iota(\{ 0 \} \times \omega_1)$ pick a representative $a' = (a'_i)_i \in B^\N$ of $f(a)$. Now for every $i \in \N$ we can choose $b_i$ to satisfy
	\[ \Phi'_i := \{ c_{a'_i} E x \colon a \in \iota(\{ 0 \} \times \omega_1) \}: \] $\Phi'_i$ is finitely realised in $\frakB$ because $\frakB \equiv \frakA$ and hence there must exist such a $b_i \in B$ by $\omega_2$-saturation of $\frakB$. It is now clear that $b := [(b_i)_{i \in \N}]$ realises the type $\Phi'$.
	
	We have therefore shown that $\frakA_\infty$ and $\frakB_\infty$ are not isomorphic. However, they are elementarily equivalent since $\frakB \succeq \frakA$.
\end{proof}

\section{Notions from games: ultraproducts as a limit of their factors}

The general usefulness of ultraproducts stems from \loz's theorem; it allows us to interpret an ultraproduct as a form of ``limit'' of its factors, at least \wrt first-order properties. In this section we want to look at the purely logical properties of ultraproducts; this will also lead us to Ehrenfeucht games.

We recall that two structures are called \emph{$m$-equivalent}, denoted by $\equiv_m$, if they satisfy the same first-order sentences of quantifier-rank up to $m$.
We also recall the following facts:
\begin{proposition}
	Let $\sigma$ be a finite relational signature and $m \in \N$.
	\begin{enumerate}
		\item There are only finitely many first-order formulae in signature $\sigma$ up to quantifier rank $m$.
		\item For every $\sigma$-structure $\frakA$ there is a sentence $\chi$ of quantifier-rank $m$ that axiomatises the $\equiv_m$-equivalence class of $\frakA$, \ie for any $\sigma$-structure $\frakB$ it holds that $\frakA \equiv_m \frakB \iff \frakB \models \chi$.
	\end{enumerate}
\end{proposition}

The following lemma is a useful starting point for investigating the behaviour of ultraproducts as limits of their factors. 

\begin{lemma}\label{lemmaUltraproductsElEquivIffMostFactorsMEquiv}
	Let $(\frakA_i)_{i \in I}$ and $(\frakB_i)_{i \in I}$ be two families of $\sigma$-structures and $\calU$ an ultrafilter on $I$. Consider the following two statements:
	\begin{enumerate}
		\item For all $m \in \N$, there are $\calU$-many indices $i \in I$ with $\frakA_i \equiv_m \frakB_i$.
		\item $\prod_i \frakA_i / \calU \equiv \prod_i \frakB_i / \calU$
	\end{enumerate}
	Then the first statement implies the second one. If the signature $\sigma$ is finite and relational, then the converse is also true.
\end{lemma}
\begin{proof}
	We first assume the first statement and prove the second one.
	Let $\phi \in \FO(\sigma)$ and let $m$ be the quantifier-rank of $\phi$. Then
	\begin{align*}
		&{\prod}_i \frakA_i / \calU \models \phi \\
		\iff &\{ i \in I \colon \frakA_i \models \phi \} \in \calU \\
		\iff &\{ i \in I \colon \frakA_i \models \phi \text{ and } \frakA_i \equiv_m \frakB_i \} \in \calU \\
		\iff &\{ i \in I \colon \frakB_i \models \phi \text{ and } \frakA_i \equiv_m \frakB_i \} \in \calU \\
		\iff &\{ i \in I \colon \frakB_i \models \phi\} \in \calU \\
		\iff &{\prod}_i \frakB_i / \calU \models \phi .
	\end{align*}
	
	Now assume conversely that the two ultraproducts $\prod_i \frakA_i / \calU$ and $\prod_i \frakB_i / \calU$ are elementarily equivalent. Let $m \in \N$. If the signature $\sigma$ is finite and relational, there is a formula $\chi$ of quantifier rank at most $m$ which axiomatises the class of all structures which are $m$-equivalent to the ultraproducts. Since the ultraproducts both satisfy $\chi$, there are $\calU$-many indices $i$ such that $\frakA_i \models \chi$ and $\frakB_i \models \chi$, which implies $\frakA_i \equiv_m \frakB_i$.
\end{proof}
\begin{remark}
	If the signature $\sigma$ is infinite, the second implication of the lemma becomes false in general. Take $\sigma = \{ c \} \cup \{ P_i \colon i \in \N \}$ with a constant symbol and a countably infinite number of unary predicate symbols and consider structures $\frakA_n$ and $\frakB_n$, all with universe $\{ 0 \}$, such that the constant symbol is interpreted by $0$ and \[ \frakA_n \models P_i c \iff i \leq 2n \quad \text{and} \quad \frakB_n \models P_i c \iff i \leq 2n + 1. \] Then none of the structures are equivalent even when we only consider quantifier-free formulae, but the ultraproducts \wrt non-principal ultrafilters are isomorphic. 
\end{remark}

In the rest of the section, we assume the signature $\sigma$ to be finite and relational.
It turns out that we can actually get even more than the preceding lemma when we consider the game-theoretic interpretation of elementary equivalence due to Ehrenfeucht and Fra\"issé. We assume that the reader is familiar with the concept of Ehrenfeucht games.
The basic result on these games is the following (\cite[Theorem 2.2.8]{finiteModelTheory}):

\begin{theorem}[Ehrenfeucht]
	The following are equivalent for any $m \in \N$ and two $\sigma$-structures $\frakA$ and $\frakB$:
	\begin{itemize}
		\item The duplicator has a winning strategy for the $m$-round Ehrenfeucht game on $\frakA$ and $\frakB$.
		\item $\frakA \equiv_m \frakB$
	\end{itemize}
\end{theorem}

As an immediate consequence, the structures $\frakA$ and $\frakB$ are elementarily equivalent iff the duplicator has a winning strategy for an Ehrenfeucht game of any pre-announced length. This is not the same as saying that the duplicator has a strategy for playing ``forever'', \ie playing a countably infinite number of rounds. If we even have his stronger condition of equivalence of $\frakA$ and $\frakB$, we also call the two structures \emph{partially isomorphic}, written $\frakA \simeq_\text{part} \frakB$. 
We will also sometimes use the extension \[ \frakA, a_1, \dotsc, a_n \simeq_\text{part} \frakB, b_1, \dotsc, b_n \] to mean that we can start the infinite Ehrenfeucht game with the elements $a_1, \dotsc, a_n$ and $b_1, \dotsc, b_n$ already selected and the duplicator still wins.

Partial isomorphy also has a characterisation in terms of logical formulae. The right logic to use this time is $L_{\infty\omega}(\sigma)$, the infinitary variant of $ \FO(\sigma)$ in which infinite disjunction and conjunction are allowed. We now get the following result (\cite[Theorem 3.2.7]{finiteModelTheory}):

\begin{theorem}[Karp]
	The following are equivalent for any $m \in \N$ and two $\sigma$-structures $\frakA$ and $\frakB$:
	\begin{itemize}
		\item $\frakA \simeq_\text{part} \frakB$
		\item $\frakA \equiv^{L_{\infty\omega}} \frakB$, \ie for all sentences $\phi \in L_{\infty\omega}(\sigma)$ we have  $\frakA \models \phi \iff \frakB \models \phi$ .
	\end{itemize}
\end{theorem}

\begin{lemma}\label{lemmaUltraproductsPartiallyIsomorphic}
	Let $(\frakA_i)_{i \in \N}$ and $(\frakB_i)_{i \in \N}$ be two sequences of $\sigma$-structures and $\calU$ a non-principal ultrafilter on $\N$ such that for all $m \in \N$ there are $\calU$-many indices $i \in \N$ with $\frakA_i \equiv_m \frakB_i$. Then the ultraproducts $\frakA := \prod_i \frakA_i / \calU$ and $\frakB := \prod_i \frakB_i / \calU$ are partially isomorphic.
\end{lemma}
\begin{proof}
	We give a winning strategy for the duplicator in the unbounded Ehrenfeucht game. The game is played with equivalence classes of sequences. Associating each equivalence class with an arbitrary representative, we can think of the game on $\frakA$ and $\frakB$ to consist of a sequence of games on the $\frakA_i$ and $\frakB_i$. For the duplicator to win the game on the ultraproducts, it suffices not to lose on $\calU$-many components at any given point in time. (For the duplicator to lose the game, there must exist an atomic formula with parameters which exhibits different behaviour on the ultraproduct; by \loz's theorem, this implies losing the game on $\calU$-many components.)
	
	For $i \in \N$ define \[ n_i := \max \{ k \leq i \colon \frakA_i \equiv_k \frakB_i \} .\] By the assumptions on $(\frakA_i)$ and $(\frakB_i)$ and non-principality of the ultrafilter $\calU$, we have $n_i \geq m$ for $\calU$-many indices $i$ when $m$ is fixed. Now the duplicator has the following winning strategy for the game on $\frakA$ and $\frakB$: When the spoiler marks the $(m + 1)$-st element in the game, she interprets this as the $(m + 1)$-st elements in the games on all $\frakA_i$ and $\frakB_i$. For indices $i$ with $n_i > m$, she answers by making a move that wins the $n_i$-round game on $\frakA_i$ and $\frakB_i$. For all other indices $i$, she just makes an arbitrary move. Then she will not have lost for one more round on $\calU$-many factors of the ultraproduct.
\end{proof}

\begin{corollary}
	Let $(\frakA_i)_{i \in \N}$ and $(\frakB_i)_{i \in \N}$ be two sequences of $\sigma$-structures and $\calU$ a non-principal ultrafilter on $\N$. Then the two ultraproducts $\frakA := \prod_i \frakA_i / \calU$ and $\frakB := \prod_i \frakB_i / \calU$ exhibit the same behaviour \wrt the logic $\mathrm{L}_{\infty\omega}$ if they are elementarily equivalent.
\end{corollary}

\begin{remark}
	Lemma \ref{lemmaUltraproductsPartiallyIsomorphic} and therefore also the preceding corollary are true for all $\omega$-saturated structures, as can readily be seen. Nevertheless, we consider our proof via a game on pairs of ultraproduct factors an interesting new approach to this standard result. 
\end{remark}

\section{Hanf sequences}

In this section, we assume the signature $\sigma$ to be finite and purely relational. 
We first recall some notions from model theory.

\begin{definition}
	Let $\frakA$ be a $\sigma$-structure with universe $A$.
	\begin{itemize}
		\item The \emph{Gaifman graph} of $\frakA$ is the graph with node set $A$ where there is an edge between two distinct nodes $a, b \in A$ iff there is a relation symbol $R \in \sigma$ and a tuple $\mathbf{a} \in R^\frakA$ such that $a$ and $b$ are elements of $\mathbf{a}$.
		\item Let $m \in \N$ and $a \in A$. The \emph{$m$-ball of $a$}, denoted $S_m(a)$, is the set of all those $a' \in A$ such that the distance between $a$ and $a'$ in the Gaifman graph of $\frakA$ is at most $m$. We also use $S_m(a)$ synonymously with the induced substructure $\frakA \restriction S_m(a)$.
		\item Let $m \in \N$ and $a \in A$. The \emph{$m$-ball type of $a$} is the isomorphism type of the structure $(\frakA \restriction S_m(a), a)$ understood as a substructure of $\frakA$ expanded by an additional constant symbol interpreted as $a$.
	\end{itemize}
\end{definition}

We note that the edge relation in the Gaifman graph is $\FO(\sigma)$-definable and therefore $m$-balls are $\FO(\sigma)$-definable. In connection with the relativisation property of first-order logic, we observe the following:
\begin{observation}
	For every formula $\phi(x) \in \FO_1(\sigma)$ and every $m \in \N$ there is a formula $\psi(x) \in \FO_1(\sigma)$ such that $\frakA, a \models \psi$ iff $S_m(a), a \models \phi$. In particular, we note that characteristic formulae axiomatising structures up to $n$-equivalence or up to isomorphism (for finite structures) can be relativised to $m$-balls in this manner.
\end{observation}

Furthermore we observe that if the Gaifman graph of a structure is disconnected, there is no relation that connects elements of different connected components. When we investigate homomorphisms and isomorphisms and also game strategies, we can therefore usually focus on one connected component:

\begin{observation}\label{observationIsomorphismFromConnComps}
	Let $\frakA$ and $\frakB$ be two $\sigma$-structures. If we can find a bijection between the connected components of the Gaifman graphs of $\frakA$ and $\frakB$ such that associated components are isomorphic, then we can construct an isomorphism of $\frakA$ and $\frakB$ by joining isomorphisms of the connected components.
\end{observation}

Gaifman graphs can be used to phrase the following standard result (originally in \cite{hanf}, quoted as in \cite[Theorem 2.4.1]{finiteModelTheory}):
\begin{theorem}[Hanf]
	Let $\frakA$ and $\frakB$ be $\sigma$-structures and let $m \in \N$. Suppose that for some $e \in \N$ all $3^m$-balls in $\frakA$ and $\frakB$ have less than $e$ elements and that for each $3^m$-ball type $\iota$, one of the following conditions holds:
	\begin{enumerate} 
		\item Both $\frakA$ and $\frakB$ have the same number of elements of $3^m$-ball type $\iota$.
		\item Both $\frakA$ and $\frakB$ have more than $m \cdot e$ elements of $3^m$-ball type $\iota$.
	\end{enumerate}
	Then the duplicator has a winning strategy in the $m$-round Ehrenfeucht game on $\frakA$ and $\frakB$, \ie $\frakA \equiv_m \frakB$.
\end{theorem}
The theorem is usually proved by explicitly giving the winning strategy. It is a useful tool for proving first-order inexpressibility of some property: it suffices to find two sequences $(\frakA_n)_n$ and $(\frakB_n)_n$ such that $\frakA_n$ is $n$-Hanf equivalent to $\frakB_n$ but the property in question is present in exactly one of the two structures.
 
We now want to present a new approach to this technique for proving inexpressibility. Our requirements will be similar to those of Hanf's theorem, but we will use ultraproducts instead of games.

\begin{definition}
	A sequence of $\sigma$-structures $(\frakA_i)_{i \in \N}$ such that each of the structures has cardinality $\leq 2^\omega$ is called a \emph{Hanf sequence} if the following conditions are satisfied:
	\begin{enumerate}
		\item For every $m \in \N$ there is an $e \in \N$ such that every $m$-ball in each of the structures contains at most $e$ elements.
		\item For every $m \in \N$ and every $m$-ball type $\iota$, one of the following two conditions holds:
			\begin{enumerate}
				\item For some $c \in \N$ there is an $i_0 \in \N$ such that for all $i \geq i_0$, $\frakA_i$ has exactly $c$ elements of $m$-ball type $\iota$.
				\item For every $c \in \N$ there is an $i_0 \in \N$ such that for all $i \geq i_0$, $\frakA_i$ contains at least $c$ elements of $m$-ball type $\iota$.
			\end{enumerate}
	\end{enumerate} 
\end{definition}

It is useful to compare the conditions for a Hanf sequence to the requirements on two structures for Hanf's Theorem: The first condition is also required there for one $m = 3^n$. The second condition in both cases can loosely be read as ``The number of elements of a specific $m$-ball type is either large or equal in all structures considered''.

\begin{remark}
	Both conditions can be rephrased. It is easy to see that for the first condition it is both necessary and sufficient to have a uniform bound for the degrees in the Gaifman graphs of all $\frakA_i$.
	
	The second condition can be rewritten in more abstract terms: the sequence $(n^\iota_i)_i$ given by 
	\[
		n^\iota_i := n^\iota(\frakA_i) := \abs{\{ a \in \frakA_i \colon \text{ $a$ has $m$-ball type $\iota$ } \}} \in \N \cup \{ \infty \}
	\]
	must converge in the topological space $\N \cup \{ \infty \}$ (given as the Alexandroff compactification of the discrete space $\N$). We will later see how we can use other notions of convergence.
\end{remark}

Our main aim for this section is the following theorem:

\begin{theorem}\label{theoremHanfSequences}
	Let $(\frakA_i)_{i \in \N}$ be a Hanf sequence. Let furthermore $\calU_1$ and $\calU_2$ be two non-principal ultrafilters on $\N$. Then the ultraproducts $\frakA^{\calU_1} := \prod_i \frakA_i / \calU_1$ and $\frakA^{\calU_2} := \prod_i \frakA_i / \calU_2$ are isomorphic.
\end{theorem}

To prove this theorem we will first look at the connected components of the Gaifman graphs of the two ultraproducts.

The connected component of an element $a$ in the Gaifman graph of a $\sigma$-structure can be written as $\bigcup_{i \in \N} S_i(a)$. We will also suggestively denote it by $S_\omega(a)$ (the \emph{$\omega$-ball of $a$}). The isomorphism type of $(S_\omega(a), a)$ will be the $\omega$-ball type of $a$. 
We note that the $\omega$-ball type of $a$ encodes the full isomorphism type of the connected component of $a$ and additionally the position of $a$ in this connected component.

\begin{lemma}\label{lemmaIsomorphyOfOmegaBalls}
	Let $\frakA$ and $\frakB$ be $\sigma$-structures. Let $a \in A$, $b \in B$ such that for each $m \in \N$ we have \[ (\frakA \restriction S_m(a), a) \simeq (\frakB \restriction S_m(b), b) \] and each of these $m$-balls is finite. Then \[ (\frakA \restriction S_\omega(a), a) \simeq (\frakB \restriction S_\omega(b), b), \] \ie $a$ and $b$ have the same $\omega$-ball type.
\end{lemma}
\begin{proof}
	The statement can be seen as an application of König's lemma: consider the set \[ M := \{ p \colon S_m(a) \to S_m(b) \colon m \in \N, \text{$p$ is an isomorphism of $(S_m(a), a)$ and $(S_m(b), b)$} \} .\] We establish a graph structure on $M$ by putting an edge between maps $p \colon S_m(a) \to S_m(b)$ and $q \colon S_{m + 1}(a) \to S_{m + 1}(b)$ iff $q$ is an extension of $p$. Arranging the vertices in $M$ in levels by the number $m$, we naturally get a tree with root $S_0(a) \to S_0(b), a \mapsto b$. (The assumption on isomorphy of the $m$-balls of $a$ and $b$ guarantees that each level is non-empty.) Since every $m$-ball is finite, there are only finitely many maps $S_m(a) \to S_m(b)$ in each level and hence every vertex has finite degree. By König's lemma, there must exist an infinite path of vertices, \ie a chain $(p_i)_{i \in \N}$ where $p_i \colon S_i(a) \to S_i(b)$ is an isomorphism and $p_{i + 1}$ is an extension of $p_i$. We can therefore form the limit $p := \bigcup_i p_i$ and easily convince ourselves that this is an isomorphism from $(S_\omega(a), a)$ to $(S_\omega(b), b)$.
\end{proof}

\begin{remark}
	Without the assumption about finiteness of the $m$-balls, the statement of the lemma can become false. In the signature of graphs consider a tree consisting of infinitely many disjoint finite paths of unbounded length starting from the root. Adding one path of infinite length then changes the isomorphism type, but does not change the isomorphism type of the $m$-balls of the root. 
\end{remark}

\begin{lemma}\label{lemmaForHanfSequences}
	Let $(\frakA_i)_{i \in \N}$ and $(\frakB_i)_{i \in \N}$ be two sequences of $\sigma$-structures, each of cardinality at most $2^\omega$, such that for every $m \in \N$ there is an $e \in \N$ such that every $m$-ball in any of the structures contains at most $e$ elements. Let furthermore $\calU_1, \calU_2$ be two non-principal ultrafilters on $\N$ such that for each $m \in \N$ and every $m$-ball type $\iota$, one of the following holds:
	\begin{enumerate}
		\item[(a)] There is a $c \in \N$ such that
			\[ \{ i \in \N \colon n^\iota(\frakA_i) = c \} \in \calU_1 \quad \text{and} \quad \{ i \in \N \colon n^\iota(\frakB_i) = c \} \in \calU_2 .\]  
		\item[(b)] For every $c \in \N$ it holds that
			\[ \{ i \in \N \colon n^\iota(\frakA_i) > c \} \in \calU_1 \quad \text{and} \quad \{ i \in \N \colon n^\iota(\frakB_i) > c \} \in \calU_2 .\]
	\end{enumerate}
	Then the ultraproducts $\frakA_\infty := \prod_i \frakA_i / \calU_1$ and $\frakB_\infty := \prod_i \frakB_i / \calU_2$ are isomorphic.
\end{lemma}
\begin{proof}
	For every $m \in \N$, there is a bound on the size of $m$-balls in $\frakA_\infty$ and $\frakB_\infty$ (The property of having $m$-balls with at most $e$ elements is first-order definable.) Therefore each $\omega$-ball in $\frakA_\infty$ is countable. Furthermore, the set of elements of $\frakA_\infty$ of specific $\omega$-ball type is given by a countable type -- by the last lemma we simply have to prescribe all $m$-ball types and each $m$-ball type is an isomorphism type of a finite structure and therefore axiomatisable by a single formula. By Proposition \ref{propCountableTypesAreRealised} this means that the number of occurrences of every $\omega$-ball type is either finite or exactly $2^\omega$. (Since $\abs{\frakA_\infty} \leq 2^\omega$, more occurrences are impossible.) If some $m$-ball subtype $\iota'$ of an $\omega$-ball type $\iota$ is realised exactly $c \in \N$ times in $\calU_1$-many $\frakA_i$, then $\iota'$ will be realised $c$ times in $\frakA_\infty$. But if this is not the case for any $c$ and any subtype of $\iota$, then there will be $2^\omega$ occurrences of $\iota$. We especially note that the number of occurrences of $\iota$ is the same in $\frakA_\infty$ and $\frakB_\infty$.
	
	We now want to show that each isomorphism type of connected components occurs the same number of times in both $\frakA_\infty$ and $\frakB_\infty$. Fix a connected component $C$ of the Gaifman graph of $\frakA_\infty$ or $\frakB_\infty$ and look for other occurrences of the isomorphism type of $C$. 
Select an $\omega$-ball type $\iota$ occurring in $C$. As $C$ is countable, $\iota$ can occur at most count\-ably many times in $C$; in all of $\frakA_\infty$ and $\frakB_\infty$, $\iota$ can only occur in connected components isomorphic to $C$. Since $\iota$ occurs either finitely often or $2^\omega$ times in both $\frakA_\infty$ and $\frakB_\infty$, the number of occurrences of $\iota$ in one of the structures determines the number of occurrences of the isomorphism type of $C$: If $\iota$ occurs $2^\omega$ times, then the isomorphism type must occur $2^\omega$ times, but if $\iota$ occurs finitely often, then we can simply compute the number of occurrences of the isomorphism type of $C$ as we know that each occurrence of the latter corresponds to a fixed number of occurrences of $\iota$. As $\iota$ occurs the same number of times in both $\frakA_\infty$ and $\frakB_\infty$, the isomorphism type of $C$ must also occur the same number of times in $\frakA_\infty$ and $\frakB_\infty$. 

It is therefore possible to match the connected components of the Gaifman graph of $\frakA_\infty$ one-to-one with the connected components of the Gaifman graph of $\frakB_\infty$ such that associated components are isomorphic.
In the light of Observation \ref{observationIsomorphismFromConnComps} this suffices to obtain an isomorphism of $\frakA_\infty$ and $\frakB_\infty$.
\end{proof}

\begin{proof}[Proof of the theorem]
	This is now immediate. Since every cofinite set is contained in every non-principal ultrafilter, the conditions on Hanf sequences guarantee the applicability of the last lemma for any pair of non-principal ultrafilters on $\N$.
\end{proof}

\begin{remark}
	We note that the slightly unwieldy conditions in the main lemma and the definition of Hanf sequences can be beautifully rephrased in topological terms. For this consider a sequence $(x_i)_{i \in \N}$ in a topological space, in this case $\N \cup \{ \infty \}$, and an arbitrary filter $\calF$ on $\N$. We define the sequence to converge to $x_\infty$ \wrt $\calF$, written $\calF{\!-\!}\lim_i x_i = x_\infty$, if for all neighbourhoods $U$ of $x_\infty$ we have \[ \{ i \in \N \colon x_i \in U \} \in \calF .\] Then the second condition on Hanf sequences simply states that $\calF_{\text{Fréchet}}{\!-\!}\lim_i n^\iota(\frakA_i)$ must exist (where $\calF_{\text{Fréchet}}$ is the cofinite filter). The condition in the main lemma reads as \[ \calU_1{\!-\!}\lim_i n^\iota(\frakA_i) = \calU_2{\!-\!}\lim_i n^\iota(\frakB_i) .\] It can be seen that both limits always exist; the remaining condition is equality of the limits. (Why do both limits exist? Consider a sequence $(x_i)$ in a compact space $X$ as a map $\mathbf x \colon \N \to X$. Then for any ultrafilter $\calU$ on $\N$, the image filter $\mathbf x(\calU)$ is an ultrafilter on $X$ and therefore converges to some $x_\infty \in X$. This is the limit of $(x_i)$ \wrt $\calU$.)
\end{remark}

\begin{corollary}
	Let $(\frakA_i)_i$ be a Hanf sequence and $m \in \N$. Then almost all $\frakA_i$ are $m$-equivalent, \ie there exists $N \in \N$ such that \[ \frakA_n \equiv_m \frakA_N \] for all $n \geq N$.
\end{corollary}
\begin{proof}
	Consider an arbitrary non-principal ultrafilter $\calU_1$ on $\N$. Since the signature $\sigma$ is finite, there is a formula $\chi$ which axiomatises $\prod_i \frakA_i / \calU_1$ up to $m$-equivalence. It now suffices to prove that only finitely many $\frakA_i$ do not satisfy $\chi$ -- all other $\frakA_i$ will then be $m$-equivalent to $\prod_i \frakA_i / \calU_1$. Assume to the contrary that there is an infinite set $I \subseteq \N$ of indices such that $\frakA_i \models \neg \chi$ for all $i \in I$. Then $I$ together with the Fréchet filter, \ie the set \[ \{ I \} \cup \{ M \subseteq \N \colon \text{ $\N \setminus M$ is finite} \} \] has the finite intersection property, hence there is a non-principal ultrafilter $\calU_2$ on $\N$ with $I \in \calU_2$. Now we have $\prod_i \frakA_i / \calU_2 \models \neg \chi$ but $\prod_i \frakA_i / \calU_1 \models \chi$ in contradiction to the last theorem.
\end{proof}

\begin{remark}
	This corollary is the primary use of the main theorem. It can also be immediately derived from Hanf's theorem. The statement itself is therefore unremarkable. However, we are unaware of any previous proof not using games. 
\end{remark}

We finish this section with some applications of the method of Hanf sequences. We will mostly consider undirected and directed graphs in the signature $\sigma_G = \{ E \}$ with binary edge relation $E$, as the transition from a structure to its Gaifman graph is very simple in these cases.

\begin{proposition}
	For $n \in \N$, let $\frakA_n$ be the directed path graph on $n$ nodes, \ie the $\sigma_G$-structure with universe $\{ 1, \dotsc, n \}$ and edge relation \[ E^{\frakA_n} = \{ (i, i + 1) \colon 1 \leq i < n \} .\] Then $(\frakA_n)_n$ is a Hanf sequence. In particular, for any given finite $m$, almost all $\frakA_n$ are $m$-equivalent. This implies that the property of having an even number of elements is not definable by a single first-order formula on this class of structures.
\end{proposition}
\begin{proof}
	Since every node has degree at most two, the first condition for Hanf sequences is easily satisfied. Let $m \in \N$ be given and consider all $m$-ball types occuring in any of the structures. For $n > 2m$, the elements $m + 1, m + 2, \dotsc, n - m$ in $\frakA_n$ all have the same $m$-ball type. All other $m$-ball types occur exactly once since they are ``cut off'' either at the upper or the lower end. This proves that $(\frakA_i)_i$ is a Hanf sequence. 
\end{proof}

\begin{proposition}
	The class of planar graphs is not definable by a first-order formula among all finite graphs.
\end{proposition}
\begin{proof}
	Let $\frakA_0 = K_5$ be the complete graph on $5$ nodes with $10$ edges. Let $\frakA_n$ be constructed from $\frakA_0$ by inserting $n$ additional nodes in each edge. This yields a graph on $5 + 10n$ nodes with $10n$ nodes of degree two and $5$ distinguished nodes of degree four such that each pair among the latter nodes is connected by a path of length $n + 1$. It is clear that $\frakA_0$ is a minor of every $\frakA_n$, hence all $\frakA_n$ are not planar.
	
	On the other hand, consider a graph $G_n$ with the following structure:
	
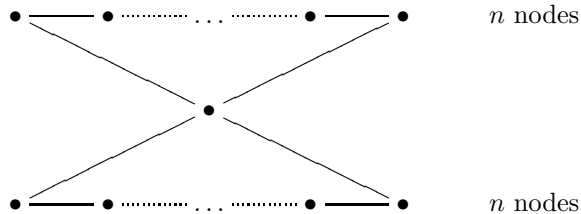
\[
	\xymatrix{
		{\bullet} \ar@{-}[r] & {\bullet} \ar@{.}[r] & {\hdots} \ar@{.}[r] &
{\bullet} \ar@{-}[r] & {\bullet} & \text{$n$ nodes} \\
		& & {\bullet} \ar@{-}[ull] \ar@{-}[dll] \ar@{-}[urr] \ar@{-}[drr] \\
		{\bullet} \ar@{-}[r] & {\bullet} \ar@{.}[r] & {\hdots} \ar@{.}[r] &
{\bullet} \ar@{-}[r] & {\bullet} & \text{$n$ nodes}
	}
	\]
	This is a graph on $2n + 1$ nodes with one node of degree four and $2n$ nodes of degree two. Let now $\frakB_n := G_n \mathop{\dot\cup} G_n \mathop{\dot\cup} G_n \mathop{\dot\cup} G_n \mathop{\dot\cup} G_n$. Then $\frakB_n$ is obviously a planar graph since each $G_n$ is planar. Furthermore, we easily convince ourselves that the sequence $\frakA_2, \frakB_2, \frakA_3, \frakB_3, \dotsc$ is a Hanf sequence. As above, it follows that planarity is not definable by a single first-order formula.	
\end{proof}

In the same manner we could now easily prove that connectivity of finite graphs is not first-order definable, as a sequence of structures consisting of cycle graphs and disjoint unions of two cycle graphs with increasing number of elements is a Hanf sequences. Instead, we prove the following stronger result. It was originally proved by Fagin in \cite{fagin} with a game argument. Nowadays it is often proved using Hanf's theorem. We slightly adapt the latter proof.
\begin{proposition}
	The property of being connected is not definable in existential monadic second order logic on the class of finite graphs.
\end{proposition}
\begin{proof}
	Assume to the contrary that the $\exists$-MSO sentence \[ \psi = \exists P_1 \dotsc \exists P_K \phi \] with $\phi \in \FO_0(\{ E, P_1, \dotsc P_K \})$ does define connectivity on finite graphs.
	For every $n \in \N$, let $\frakC_n$ be a cycle graph coloured in such a way as to satisfy $\phi$, \ie a $\{ E, P_1, \dotsc P_K \}$-structure satisfying $\phi$ in which every element has degree $2$, and let $\frakC_n'$ be a a disjoint union of two coloured cycle graphs such that the following condition is satisfied:
	\[ \text{For every $n$-ball type $\iota$, both $\frakC_n$ and $\frakC_n'$ have the same number of elements of $n$-ball type $\iota$.} \]
	We will need to argue later that this condition can be satisfied. Note that $\frakC_n'$ cannot satisfy $\phi$ since a disjoint union of cycle graphs is not connected. We could now consider the sequence of structures $\frakC_1, \frakC_1', \frakC_2, \frakC_2', \dotsc$, but it is not clear whether this is a Hanf sequence. The original Lemma \ref{lemmaForHanfSequences} comes to the rescue. Let $\calU$ be a non-principal ultrafilter on $\N$ and consider $\frakC_\infty := \prod_i \frakC_i / \calU$ and $\frakC_\infty' := \prod_i \frakC_i' / \calU$. By construction $\frakC_\infty \models \phi$ and $\frakC_\infty' \models \neg\phi$.
	
	However, for all $m \in \N$ and all $m$-ball types $\iota$ we have $n^\iota(\frakC_k) = n^\iota(\frakC_k')$ whenever $k \geq m$. From this we quickly derive \[ \calU{\!-\!}\lim_i n^\iota(\frakC_i) = \calU{\!-\!}\lim_i n^\iota(\frakC_i') .\] By Lemma \ref{lemmaForHanfSequences}, this implies $\frakC_\infty \simeq \frakC_\infty'$, which contradicts the behaviour of $\phi$.
	
	It remains to see that we can find $\frakC_n$ and $\frakC_n'$ with the desired condition. This is exactly as in the standard proof of this proposition using Hanf's theorem. Consider a cycle graph on $N \geq 2n + 2$ nodes. As this satisfies $\psi$, there exists a colouring such that the coloured graph satisfies $\phi$. Since the $n$-ball of an element has exactly $2n + 1$ elements, at most $(2^K)^{2n + 1}$ $n$-ball types occur. For large $N$ we can force the existence of at two elements $a, b$ having the same $m$-ball type, even when taking cycle orientation into account, and having distance at least $2n + 1$. We can now modify two edges starting at $a$ and $b$ such that the cycle splits into two cycles, but without modifying the $m$-ball type of any element: Let $a'$ and $b'$ be neighbours of $a$ and $b$, respectively, in the same direction, \ie a traversal of the cycle starting at $a$ and $a'$ yields the nodes in order $a, a', b, b'$. Removing the edges $\{ a, a' \}$ and $\{ b, b' \}$ and adding the edges $\{ a, b' \}$ and $\{ b', a' \}$ results in two disjoint cycles, but does not change $n$-ball types due to choice of $a$ and $b$. Putting $\frakC_n$ to be the original coloured cycle of length $N$ and $\frakC_n'$ to be the modified graph fulfills the desired conditions.
	
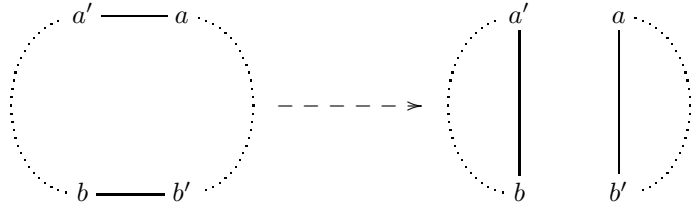
\[
	\xymatrix{
		a' \ar@(l,l)@{.}[dd] & a \ar@{-}[l] & & & & a' \ar@(l,l)@{.}[dd] & a \ar@{-}[dd]\\
		& & \ar@{-->}[rr] & & \\
		b \ar@{-}[r] & b' \ar@(r, r)@{.}[uu] & & & & b \ar@{-}[uu] & b' \ar@(r, r)@{.}[uu]
	}
	\]
\end{proof}

\section{Gaifman's theorem via ultraproducts}

We again restrict our attention to finite relational signatures in this section.

Besides Hanf's theorem there is one more well-known locality theorem in first-order logic, which is due to Gaifman.
\begin{definition}
	A formula $\phi(x) \in \FO_1(\sigma)$ is \emph{$l$-local} for $l \in \N$ if for all $\sigma$-structures $\frakA$ and all $a \in A$ we have \[ \frakA, a \models \phi \iff \frakA \restriction S_l(a), a \models \phi .\]
	A sentence of the form 
		\[ \exists x_1, \dotsc, x_n \Big( \bigwedge_{1 \leq i < j \leq n} d_{> 2l}(x_i, x_j) \land \bigwedge_{1 \leq i \leq n} \psi(x_i) \Big) \]
	where $\psi(x) \in \FO_1(\sigma)$ is an $l$-local formula is called a \emph{basic local sentence}. Here $d_{>2l}(x_i, x_j)$ is to be understood as a first-order formula asserting that $x_i$ and $x_j$ have distance greater than $2 l$ in the Gaifman graph.
\end{definition}

These definitions are useful because of the following locality result (\cite{gaifman}):
\begin{theorem}[Gaifman]
	Every first-order sentence is logically equivalent to a boolean combination of basic local sentences.
\end{theorem}

This result is frequently proved by an Ehrenfeucht-Fra\"issé argument, although this is not how Gaifman originally proved it. We want to prove a variant of the theorem using ultraproducts. In our version we need an additional constraint on the size of $m$-balls as we know it from Hanf's theorem. We will also prove the full statement of Gaifman's theorem using a combination of game techniques, ultraproducts and compactness.

For both versions of the theorem we need the following consequence of the saturation properties of ultraproducts.

\begin{lemma}\label{lemmaOmegaSaturationImpliesNTypesAreRealised}
	Let $\frakA$ be $\omega$-saturated and $\Psi \subseteq \FO_n(\sigma)$ such that \[ \frakA \models \exists x_1 \dotsm \exists x_n \bigwedge \Psi_0(x_1, \dotsc, x_n) \] for all finite subsets $\Psi_0 \subseteq \Psi$. Then there are elements $a_1, \dotsc, a_n \in A$ such that \[ \frakA, a_1, \dotsc, a_n \models \Psi .\]
\end{lemma}

We can now prove our first version of Gaifman's theorem.

\begin{theorem}\label{theoremGaifmanVariantIsomorphy}
	Let $\frakA, \frakB$ be two $\N$-fold ultraproducts of cardinality $\leq 2^\omega$ \wrt non-principal ultrafilters such that the following condition is satisfied: For each $m \in \N$ there is an $e \in \N$ such that every $m$-ball in $\frakA$ and $\frakB$ has at most $e$ elements. If $\frakA$ and $\frakB$ satisfy the same basic local sentences, then they are isomorphic.
\end{theorem}
\begin{proof}
	As in the proof of Theorem \ref{theoremHanfSequences} on Hanf sequences we want to show that each isomorphism type of connected components occurs as many times in $\frakB$ as it occurs in $\frakA$. By Observation \ref{observationIsomorphismFromConnComps} this suffices to prove that the structures are globally isomorphic. We recall that each connected component is described by any of the $\omega$-ball types of elements in it. Since each connected component is countable by the size restriction on $m$-balls and each $\omega$-ball type is either realised finitely often or exactly $2^\omega$ times, we find that each isomorphism type of connected components occurs finitely often or exactly $2^\omega$ times.
	
	Let now $\iota$ be an arbitrary $\omega$-ball type. Assume that there are $n$ disjoint connected components in the Gaifman graph of $\frakA$ in which $\iota$ occurs; we show that there are at least $n$ such components in $\frakB$. Since the situation is symmetric in $\frakA$ and $\frakB$, this suffices to prove the claim.
	Let $a_1, \dotsc, a_n \in A$ with $\omega$-ball type $\iota$ in disjoint connected components of the Gaifman graph of $\frakA$. For $m \in \N$ let $\chi_m(x) \in \FO_1(\sigma)$ be a formula asserting that \[ (S_m(x), x) \simeq (S_m(a_1), a_1) .\] Such a formula exists because $S_m(a_1)$ is finite by assumption. By definition $\chi_m$ is $m$-local. Now consider the sequence $(\psi_m)_m$ of formulae in $n$ free variables $x_1, \dotsc, x_n$ given by \[ \psi_m := \bigwedge_{1 \leq i < j \leq n} d_{>2m}(x_i, x_j) \land \bigwedge_{1 \leq i \leq n} \chi_m(x_i) .\] Clearly we have \[ \frakA, a_1, \dotsc, a_n \models \psi_m \] for all $m$. Hence each of the basic local sentences \[ \exists x_1 \dotsm \exists x_n \psi_m \] is satisfied in $\frakA$ and therefore also in $\frakB$. Since $\psi_{m + 1} \models \psi_m$, this means that every formula of the form \[ \exists x_1 \dotsm \exists x_n \bigwedge_{m \leq M} \psi_m \] with $M \in \N$ is true in $\frakB$. By the last lemma we therefore find elements $b_1, b_2, \dotsc, b_n$ such that $\frakB, b_1, \dotsc, b_n \models \psi_m$ for all $m$. The $b_i$ must pairwise have infinite distance since $\psi_m$ asserts that they have distance greater than $2 m$. Furthermore each $b_i$ has $\omega$-ball type $\iota$ by Lemma \ref{lemmaIsomorphyOfOmegaBalls}.
	We have therefore found $n$ disjoint connected components of $\frakB$ in which the $\omega$-ball type $\iota$ occurs.
\end{proof}

We note that in the proof we did not actually use the factor structures of the ultraproducts. We only used the saturation properties of ultraproducts proved in the first section.

As a corollary we get a weak version of the Keisler-Shelah isomorphism theorem (\cite{keisler, shelahForKeislerShelah}), which in its full version is notoriously difficult to prove.
\begin{corollary}
	Let $\frakA$ and $\frakB$ be elementarily equivalent $\sigma$-structures of cardinality $\leq 2^\omega$ such that for every $m \in \N$ there is an $e \in \N$ such that every $m$-ball in $\frakA$ and $\frakB$ has at most $e$ elements.
	Let $\calU_1$ and $\calU_2$ be two arbitrary non-principal ultrafilters on $\N$. Then we have \[ \frakA^\N / \calU_1 \simeq \frakB^\N / \calU_2 .\]
\end{corollary}
\begin{proof}
	Both the cardinality requirements and the other conditions expressible as first-order statements on $\frakA$ and $\frakB$ translate to their ultrapowers. Then the previous theorem is applicable.
\end{proof}

We are now almost in a position to prove the usual version of Gaifman's theorem. If we can get rid of the size restrictions on $m$-balls in Theorem \ref{theoremGaifmanVariantIsomorphy}, a simple compactness argument will yield Gaifman's theorem. We now join forces with the game technique to prove the following two lemmas. These are analogous to Lemma \ref{lemmaIsomorphyOfOmegaBalls} and Observation \ref{observationIsomorphismFromConnComps}.

\begin{lemma}\label{lemmaElementaryEquivalenceOfOmegaBalls}
	Let $\frakA, \frakB$ be two $\omega$-saturated structures in signature $\sigma$ and $a \in A$, $b \in B$ such that
	\[ (S_k(a), a) \equiv (S_k(b), b) \] for all $k \in \N$. Then \[ S_\omega(a), a \simeq_\text{part} S_\omega(b), b .\] 
\end{lemma}
\begin{proof}
	We give a winning strategy for the duplicator in the infinite Ehrenfeucht game. During the game we always maintain the following invariant: When $m$-tuples of elements $a_1, \dotsc, a_m \in S_\omega(a)$ and $b_1, \dotsc, b_m \in S_\omega(b)$ have been played after $m$ rounds, it holds that \[ (S_k(a), a, a_1, \dotsc, a_m) \equiv (S_k(b), b, b_1, \dotsc, b_m) \] for all $k$ such that $a_1, \dotsc, a_m \in S_k(a)$. If we manage to maintain this condition, it is clear that the duplicator will never lose the game.
	
	Assume that $m$ rounds have been played. Let the spoiler choose a new element $a_{m + 1} \in S_\omega(a)$. (By symmetry of the situation it does not matter whether the spoiler chooses an element in $S_\omega(a)$ or in $S_\omega(b)$.)  For $n \in \N$ and sufficiently large $k$, let $\phi_{k, n} \in \FO_{m + 2}(\sigma)$ be a formula such that
	\[ \mathfrak{C}, c, c_1, \dotsc, c_{m + 1} \models \phi_{k, n} \iff (S_k(c), c, c_1, \dotsc, c_{m + 1}) \equiv_n (S_k(a), a, a_1, \dotsc, a_{m + 1}) \]
	for any $\sigma$-structure $\mathfrak{C}$ and elements $c, c_1, \dotsc, c_{m + 1}$; this is possible as we can axiomatise the ball $(S_k(a), a, a_1, \dotsc, a_n)$ up to $n$-equivalence with a single formula since the signature $\sigma$ is finite. To maintain our invariant it suffices to find $b_{m + 1} \in S_\omega(b)$ such that \[ \frakB, b, b_1, \dotsc, b_m, b_{m + 1} \models \phi_{k, n} \] for all $k$ and all $n$. But we can interpret the set $\Phi$ consisting of all $\phi_{k, n}$ as a countable type with parameters $b, b_1, \dotsc, b_m$ since $\Phi$ is finitely realised in $\frakB$: Let $\Phi_0 \subseteq \Phi$ be a finite subset with the value $k$ of all $\phi_{k, n}$ in it bounded by some $K$. Then $a_{m + 1}$ realises $\Phi_0$ in $S_K(a)$ where we use $a, a_1, \dotsc, a_m$ as parameters; hence \[ \exists x \bigwedge\Phi_0(b, b_1, \dotsc, b_m, x) \] is also true in $S_K(b)$ (and therefore in $\frakB$) because of elementary equivalence. 
	 
	 Because of $\omega$-saturation we therefore find $b_{m + 1} \in B$ realising $\Phi$ with parameters $b, b_1, \dotsc, b_m$. Since the set of all $\phi_{kn}$ also fixes the distance of $b$ and $b_{m + 1}$, we even get $b_{m + 1} \in S_\omega(b)$. This establishes the invariant for one more round.
\end{proof}

\begin{remark}
	Readers familiar with Fra\"issé's algebraic characterisation of partial isomorphy will recognise that we have essentially proved the existence of a back \& forth system.
\end{remark}

\begin{lemma}
	Let $\frakA$ and $\frakB$ be $\sigma$-structures. Assume that for any finite number of disjoint connected components $A_1, \dotsc, A_n$ of the Gaifman graph of $\frakA$ with \[ \frakA \restriction A_1 \equiv \dotsb \equiv \frakA \restriction A_n \] there is an equal number $B_1, \dotsc, B_n$ of disjoint connected components of the Gaifman graph of $\frakB$ with \[ \frakA \restriction A_1 \equiv \frakB \restriction B_1 \equiv \dotsb \equiv \frakB \restriction B_n \] and vice versa with $\frakA$ and $\frakB$ exchanged; in other words, each $\equiv$-class of connected components either occurs the same finite number of times in both $\frakA$ and $\frakB$ or occurs infinitely often in both structures.
	
	In this case, $\frakA$ and $\frakB$ are elementarily equivalent.
\end{lemma}
\begin{proof}
	There is a winning strategy for the duplicator in an $m$-round Ehrenfeucht game on $\frakA$ and $\frakB$ for any $m \in \N$: For any move of the spoiler in a connected component of $\frakA$ or $\frakB$ in which no element has been picked yet, the duplicator chooses a new elementarily equivalent connected component of $\frakB$ or $\frakA$, respectively, and plays as if an $m$-round game were played on the two connected components. If the spoiler picks an element of a connected component of $\frakA$ or $\frakB$ in which an element has been picked before, the duplicator continues with the sub-game on the associated elementarily equivalent component of $\frakB$ or $\frakA$, respectively. Since we can ensure that no sub-game is lost by the duplicator -- we only select pairs of elementarily equivalent connected components -- the duplicator will not lose the global game.
\end{proof}

As the main ingredient for the proof of Gaifman's theorem we use the following lemma. Here we don't even need $\frakA$ and $\frakB$ to satisfy the specific ultraproduct saturation properties (e.g.\ that a countable type is realized by either finitely many or at least $2^\omega$ many elements), simple $\omega$-saturation suffices.

\begin{lemma}
	Let $\frakA, \frakB$ be two $\omega$-saturated $\sigma$-structures. If $\frakA$ and $\frakB$ satisfy the same basic local sentences, then they are elementarily equivalent.
\end{lemma}
\begin{proof}
	The proof is similar to the proof of Theorem \ref{theoremGaifmanVariantIsomorphy} except that elementary equivalence and partial isomorphy replace full isomorphy.
	
	We want to apply the last lemma. Since partial isomorphy is a finer equivalence relation than elementary equivalence, it suffices to show that each $\simeq_\text{part}$-class of connected components either occurs the same finite number of times in both $\frakA$ and $\frakB$ or occurs infinitely often in both. Therefore choose $n$ partially isomorphic disjoint connected components $A_1, \dotsc, A_n$ of $\frakA$ and show that there are at least $n$ connected components of $\frakB$ of the same partial isomorphism type. (The situation is symmetric in $\frakA$ and $\frakB$.)
	
	We pick an arbitrary element $a_1 \in A_1$. By playing the first round of an infinite Ehrenfeucht game, we find elements $a_i \in A_i$ such that \[ A_1, a_1 \simeq_\text{part} A_2, a_2 \simeq_\text{part} \dotsb \simeq_\text{part} A_n, a_n .\] For $k, m \in \N$ let $\chi_{k, m}(x) \in \FO_1(\sigma)$ be a formula asserting that \[ (S_k(x), x) \equiv_m (S_k(a_1), a) .\] This is obviously a $k$-local formula. We now consider the family $(\psi_{k, m})_{k, m \in \N}$ of formulae in $n$ free variables $x_1, \dotsc, x_n$ given by \[ \psi_{k, m} := \bigwedge_{1 \leq i < j \leq n} d_{> 2m}(x_i, x_j) \land \bigwedge_{1 \leq i \leq n} \chi_{k, m}(x_i) .\] Since each of the sentences \[ \exists x_1 \dotsm \exists x_n \bigwedge\Psi_0(x_1, \dotsc, x_n), \] where $\Psi_0$ is a finite collection of formulae of the form $\psi_{k, m}$, is equivalent to a basic local sentence and satisfied in $\frakA$ (by $a_1, \dotsc, a_n$) and therefore also in $\frakB$, Lemma \ref{lemmaOmegaSaturationImpliesNTypesAreRealised} implies the existence of elements $b_1, \dotsc, b_n$ such that \[ \frakB, b_1, \dotsc, b_n \models \psi_{k, m} \] for all $k, m$. This means that the $b_i$ have pairwise infinite distance, \ie are in different connected components of $\frakB$, and each satisfy all $\chi_{k, m}$, implying \[ (S_k(b_i), b_i) \equiv (S_k(a_1), a_1) \] and therefore, by Lemma \ref{lemmaElementaryEquivalenceOfOmegaBalls}, \[ (S_\omega(b_i), b_i) \simeq_\text{part} (S_\omega(a_1), a_1) .\]
	We have therefore found $n$ disjoint connected components of $\frakB$ of the same partial isomorphism type as the $A_i$.
\end{proof}

We can now use compactness to prove Gaifman's theorem.

\begin{proof}[Proof of Gaifman's theorem]
	Let $\phi \in \FO_0(\sigma)$ be a sentence. Let 
	\begin{align*} 
		\Psi &:= \{ \psi \in \FO_0(\sigma) \colon \text{ $\psi$ is a boolean combination of basic local sentences} \}, \\
		\Psi_\phi &:= \{ \psi \in \Psi \colon \phi \models \psi \} 
	.
	\end{align*} By compactness it suffices to show that $\Psi_\phi \models \phi$ since we then have $\bigwedge \Psi^0 \equiv \phi$ for some finite subset $\Psi^0 \subseteq \Psi_\phi$. Assume to the contrary that $\Psi_\phi \cup \{ \neg \phi \}$ has a model $\frakA$. Let \[ \Psi_\frakA := \{ \psi \in \Psi \colon \frakA \models \psi \} ;\] we claim that $\Psi_\frakA \cup \{ \phi \}$ is satisfiable. If it were not satisfiable, there would be a finite subset $\Psi_\frakA^0 \subseteq \Psi_\frakA$ such that $\phi \models \neg \bigwedge \Psi_\frakA^0$, but then $\neg \bigwedge \Psi_\frakA^0 \in \Psi_\phi$ and therefore $\frakA \models \neg \bigwedge \Psi_\frakA^0$, which is a contradiction to $\frakA \models \Psi_\frakA$. Hence we find a model $\frakB$ of $\Psi_\frakA \cup \{ \phi \}$.
	
	Now the two structures $\frakA$ and $\frakB$ satisfy the same basic local sentences, but disagree on $\phi$. Forming $\N$-fold ultrapowers (or getting $\omega$-saturated elementary extension in some other way) we get a contradiction to the last lemma.
\end{proof}

\begin{remark}
	This new proof for Gaifman's theorem has been found independently by Lindell, Towsner and Weinstein (\cite[Theorem (11)]{lindellWeinstein}).
\end{remark}

\section{Ultraproducts of finite linear orderings}

In the preceeding sections we have been reasonably successful in proving inexpressibility results for some classes of finite structures, especially graphs, using variants of the locality theorems of Hanf and Gaifman. In this section we look at one more class of structures, namely finite linear orderings. The results obtained in the previous two sections are not applicable as the Gaifman graph of a linearly ordered structure is a complete graph -- locality is therefore not a useful concept.

As a starting point, we look at the following theorem. It will turn out that we can use it to prove more general inexpressibility results.

\begin{theorem}\label{theoremEvenNumberOfEltsIsNotDefinableForLinOrds}
	The property of having an even number of elements is not $\FO(\{ < \})$-definable in the class of all finite linear orderings.
\end{theorem}

To show this, let $\frakA_n$ be the canonical linear ordering on the elements $\{ 1, \dotsc, n \}$. Let $\calU$ be an arbitrary non-principal ultrafilter on $\N$. To prove the theorem, it suffices to show that \[ \frakA_\text{even} := \prod_{n \in \N} \frakA_{2n + 2} / \calU \equiv \prod_{n \in \N} \frakA_{2n + 1} / \calU =: \frakA_\text{odd}, \] since a single formula axiomatising even cardinality would separate $\frakA_\text{even}$ from $\frakA_\text{odd}$. It turns out that we can prove an even stronger statement with a simple structural argument.

\begin{lemma}
	The two structures $\frakA_\text{even}$ and $\frakA_\text{odd}$ are isomorphic.
\end{lemma}
\begin{proof}
	Both structures are discrete linear orderings with minimal and maximal elements since these are first-order definable properties. Consider the element \[ a_\text{max} := [ (2n + 2)_{n \in \N} ] \in \frakA_\text{even} ;\] it is clearly maximal in $\frakA_\text{even}$. As any element $[(a_n)_n]$ of $\frakA_\text{even}$ not equal to $a_\text{max}$ has a representative such that $a_n < 2n + 2$ for every $n$, we get a canonical isomorphism $\frakA_\text{odd} \simeq \frakA_\text{even} \setminus \{ a_\text{max} \}$. But since $\frakA_\text{even}$ is infinite (due to non-principality of $\calU$), its order type can be written as an ordered sum of a copy of $\omega$, some copies of $\omega^\ast + \omega$ and a copy of $\omega^\ast$. ($\omega^\ast$ denotes the reverse order type of $\omega$.) Removing any element clearly does not change the order type, \ie \[ \frakA_\text{even} \simeq \frakA_\text{even} \setminus \{ a_\text{max} \} \simeq \frakA_\text{odd}. \qedhere \]
\end{proof}

\begin{corollary}
	Let $m \in \N$. Then there is an $N \in \N$ such that $\frakA_{2n + 1} \equiv_m \frakA_{2n + 2}$ for all $n \geq N$.
\end{corollary}
\begin{proof}
	Assume there were infinitely many $n$ such that $\frakA_{2n + 1} \not\equiv_m \frakA_{2n + 2}$. Then we could construct a non-principal ultrafilter containing the set of all such $n$. In the light of the preceding lemma, this is a contradiction to Lemma \ref{lemmaUltraproductsElEquivIffMostFactorsMEquiv}.
\end{proof}

\begin{corollary}
	Let $m \in \N$. Then almost all $\frakA_n$ are $m$-equivalent, \ie there exists $N \in \N$ such that \[ \frakA_n \equiv_m \frakA_N \] for all $n \geq N$.
\end{corollary}
\begin{proof}
	We can repeat the proof of the lemma to get
		\[ \prod_n \frakA_{2n + 2} / \calU \simeq \prod_n \frakA_{2n + 3} / \calU .\]
	As in the preceding corollary, this yields $\frakA_{2n + 2} \equiv_m \frakA_{2n + 3}$ for almost all $n$. Combining this with the preceding corollary we get $\frakA_n \equiv_m \frakA_{n + 1}$ for all $n$ beyond some finite $N$.
\end{proof}

\begin{remark}
	This last corollary is a much more general inexpressibility result than the theorem at the start of the section. It is also a standard result in first-order logic, which is usually obtained by giving a winning strategy for the duplicator in an $m$-round Ehrenfeucht game. Since locality theorems are not applicable -- the Gaifman graph of a linear ordering is always a complete graph -- this is normally done explicitly. We did not need games at all and used ultraproducts instead.
\end{remark}

We note that our results for finite directed path graphs are stronger than those for finite linear orderings. We showed that the sequence of finite directed path graphs is a Hanf sequence, therefore any two ultraproducts \wrt non-principal ultrafilters are isomorphic. For finite linear orderings, we have only shown that the ultraproduct does not change when we shift the ultrafilter by one position (corresponding to an ultrafilter on the odd numbers versus one on the even numbers). Only under the continuum hypothesis do we get isomorphy for two arbitrary ultrafilters due to \ref{corollaryElemEquivUltraprodsAreIsomAssumingCH}.
	
	This difference between finite linear orderings and directed path graphs should not surprise us as it is a familiar observation that a successor relation is first-order definable from an ordering but not vice versa. When we leave the finite by means of an ultraproduct (preserving not much more than first-order properties), the natural correspondence of successor structure and ordering breaks down.

\section{Conclusion}

We have explored the ultraproduct method as an independent approach to first-order inexpressibility results. We have proved well-known inexpressibility results for finite linear orderings and finite graphs without resorting to either compactness or Ehren\-feucht-Fra\"issé arguments. Only for our new proof of Gaifman's theorem did we need games. Hence we have found ultraproducts to be an unexpectedly useful tool for the investigation of first-order-properties in finite model theory.

Some obvious structural questions about ultraproducts could not be answered by this paper. We have already seen in the first section that properties of ultraproducts may depend on the continuum hypothesis. It would be desirable to further investigate in which ways elementarily equivalent $\N$-fold ultraproducts of cardinality $2^\omega$ may fail to be isomorphic and whether such failures must always occur when we assume the negation of the continuum hypothesis. Furthermore, we would like to see an example for a single structure such that the isomorphism types of its $\N$-fold ultrapowers depend on the choice of non-principal ultrafilters. Lastly, we would like to see whether there are any ``natural'' examples for non-isomorphic elementarily equivalent ultraproducts -- for example, a closer investigation of the ultraproducts of finite linear orderings would be helpful, as in the last section we could only show elementary equivalence of different ultraproducts.

\bibliographystyle{alpha}
\bibliography{Ultraproducts}

\end{document}